\documentclass[12pt, reqno]{amsart}
\usepackage{amsmath, amsthm, amscd, amsfonts, amssymb, graphicx, color}

\usepackage{calrsfs}

\newtheorem{theorem}{Theorem}[section]

\theoremstyle{definition}

\theoremstyle{remark}

\numberwithin{equation}{section}

\def\vphi{\varphi}
\def\Acal{\mathcal A}
\def\Tcal{\mathcal T}
\def\Vcal{\mathcal V}

\begin{document}
\setcounter{page}{1}

\title[Independence]{On the dependence of conditions\\ in the linear mapping definition}

\author[Aslanbek Naziev]{Aslanbek Naziev$^1$}

\address{$^{1}$ Department of Mathematics, Ryazan State University, Svobody str., 46, Ryazan, 390000, Russian Federation; Ryazan Institute for Education Development, Firsova str., 2a, Ryazan, 390023, Russian Federation; State University for Humanities and Social Studies, Zeleonaya str 30, Kolomna, Moscow reg., 140411, Russian Federation.}
\email{a.naziev@365.rsu.edu.ru}




\begin{abstract}
We study the (in)dependence of additivity and homogeneity conditions in the definition of linear mappings between vector spaces over the same scalar field. Unlike other works on the subject, dealing with particular fields like real or complex numbers, or with particular mappings like continuous or measurable, we consider the general case. This enables us to obtain complete picture. Namely, for the prime field, and only in this case, the conditions are dependent (additivity implies homogeneity). For the non-prime field they are independent: neither of conditions implies the other.
\end{abstract}

\maketitle

\section{Introduction and preliminaries}

\noindent By definition, a statement $\vphi$ is independent of a set of statements $\Acal$ in a theory $\Tcal$ if neither $\vphi$ nor its negation $\lnot\vphi$ is a consequence of $\Acal$ in $\Tcal$. In case $\Acal=\{\psi\}$, they say that statements $\vphi$ and $\psi$ are independent in theory $\Tcal$. Due to the contraposition law, this means that neither of $\vphi$, $\psi$ implies the other in theory~$\Tcal$.

In the XX century, the problems of independence attracted sufficient large attention. Great mathematicians K.~G\"{o}del and P.~Cohen proved very hard theorems about the independence of the axiom of choice and continuum hypothesis from other axioms of set theory (\cite{God}, \cite{Coh}, see also \cite{Kun}, \cite{Ros}). There were also proven results of another type, namely, of consistency (see, for example, \cite{Spe}). Nevertheless, some problems looking simple remain unsolved. Among such problems is the (in)dependence of conditions in the definition of a linear mapping.

Take as $\Tcal$ the theory $\Vcal(F)$ of vector spaces over one and the same scalar field~$F$. Let $U$ and $V$ be vector spaces over $F$. As known, a mapping $\varphi\colon U\to V$ is linear iff it satisfies two conditions;
\begin{enumerate}
\item[(A)]
For all $u_1$, $u_2$ from $U$, $f(u_1+u_2)=f(u_1)+f(u_2)$ (additivity);
\item[(H)]
For every $u$ from $U$ and every $\lambda$ from $F$, $\varphi(\lambda u) = \lambda\varphi(u)$ (homogeneity).
\end{enumerate}
In the light of previous remarks, the following problem arises: are the conditions (A) and (H) (in)dependent in theory $\Vcal(F)$?

To some, the question may seem strange or uninteresting, but this impression is deceptive. This question arises naturally every time when one needs to verify whether given mapping is linear. In such cases, one needs to make sure that BOTH conditions that make up the definition of a linear mapping are satisfied. But if one of the conditions implies the other, then it is preferable to establish only one condition and get another condition as a consequence of the first one. So naturally the questions arise, ``Maybe, (A) implies (H)?'' or, ``Maybe, (H) implies (A)?''

In this regard, I must emphasize that there is absolutely no reason to consider either or both of these questions trivial. To understand this, just look at the semantic structure of conditions. Condition (A) includes two vector variables, while (H) includes one vector and one scalar variable. Where can hopes for the triviality of connections between such different conditions come from? Indeed, the proofs of the corresponding theorems given below, although not very complicated, are not so trivial either.

The equation
\[
  f(a+b) = f(a) + f(b)
\]
in the context of functions of real or complex variables is known as the Cauchy (functional) equation. So, additivity condition can be reasonable named the generalized Cauchy (functional) equation. In the theory of Cauchy equation naturally arises the question about properties of the Cauchy equation solutions. From this point of view, the question,
\begin{quote}
Does additivity imply homogeneity?
\end{quote}
get the interpretation,
\begin{quote}
Do the solutions of Cauchy equation have the homogeneity property?
\end{quote}

From the other hand, homogeneity condition arises in many contexts and, among them, in the context of ordinary differential equations theory. In this context the question whether homogeneity implies additivity appears as a question about properties of solutions of homogeneous ODE's. Problems about these properties are ordinary in the courses of ODE's.

Transmission of these questions to the linear algebra reveals an interesting fact. In~\cite{Brad} is described one case in which students of a linear algebra class expressed the confidence that homogeneity implies additivity. My teaching practise shows that this is quite ordinary occurrence. The explanation of this phenomenon lies in the practise of teaching school mathematics. In this practice linear function --- that is the function of the form $y=ax+b$. At the same time the function of the form $y=kx$ is named homogeneous (linear function). By this reason, the word `homogeneous' in the mind of most students is associated with the functions of the form $y=kx$. Because these functions have the property of additivity, most students are sure that homogeneity (evidently to them!) implies additivity. This example already shows how important is to explicitly consider and express the interrelation between the two properties.

Despite the seeming simplicity of the problem, it remained, until now, somewhat unsolved. As we already have said, traditionally, the question ``whether (A) implies (H)?'' is considered in the context of Cauchy equation theory, predominantly in the cases $F=\mathbb R$ or $\mathbb C$ under some additional assumptions like continuity (see, for example, \cite{Acz, Brad, Kuc, JMat}). The reverse question ordinarily arises in the context of ODE in exercises connected with the notion of homogeneous equation.

Contrary to these works, we consider here the general case of arbitrary field $F$ obtaining the following complete picture. Namely, in the case of prime $F$, and only in this case, additivity implies homogeneity, while homogeneity never imply additivity.

In other words,
\begin{itemize}
\item for the prime field $F$, (A) implies (H);
\item for the non-prime field $F$, (A) not implies (H);
\item for every field, whether prime or non-prime, (H) not implies (A).
\end{itemize}
So, for the prime field $F$, the conditions (A) and (H) are dependent. For the non-prime field, they are independent.

\section{Main results}
\subsection{From (A) to (H)}
Recall (see, e. g., {\cite{MB}}) that the prime subfield of a field $F$ is the smallest subfield of $F$, and that a field $F$ is a prime field if it coincides with its prime subfield. Recall also that every prime field is isomorphic to $\mathbb Q$ or ${\mathbb Z}_p (={\mathbb Z}/p{\mathbb Z})$ for some prime number $p$.

\begin{theorem}
Let $F$ be a field, and $k$ its prime subfield.

If $F=k$, then additivity implies homogeneity: every additive mapping $U\to V$ with vector $F$-spaces $U$, $V$,
is $F$-homogeneous (and therefore $F$-linear).

If $F\ne k$, then additivity does not imply homogeneity: there exist vector $F$-spaces $U$ and $V$ and additive mapping
$U\to V$, which is not $F$-homogeneous (and therefore not $F$-linear).
\end{theorem}

\begin{proof}
Let $F=k$ and $\varphi\colon U\to V$ be additive mapping of the vector $F$-space $U$ to vector $F$-space $V$. Then well known arguments show that $\varphi$ is $F$-homogeneous.

[For the sake of completeness, recall these arguments. In case $F=k={\mathbb{Z}}_p$ for some prime number $p$, every $\lambda\in F$ has form $\bar{m}$ with $m\in\{1,\ 2,\ldots, p\}$, so
\[
  \varphi(\lambda x) = \varphi(\underbrace{x+\ldots+x}_{m\ \text{times}})= \underbrace{\varphi(x)+\ldots+\varphi(x)}_{m\ \text{times}}=\lambda\varphi(x).
\]

In case $F=k={\mathbb{Q}}$, for every $\lambda\in F$ there exist $m,\ n\in\mathbb{Z}$ such that $n\ne 0$ and $\lambda = \frac{m}{n}$. Due to this we have, first, for every $m$ and $n$,
\begin{equation}\nonumber
  \varphi(\lambda x) = \varphi\Bigl(\frac{m}{n} x\Bigr) = \varphi\Biggl(\Bigl(\underbrace{\frac{1}{n}+\ldots\frac{1}{n}}_{m\ \text{times}}\Bigr)x\Biggr) = \underbrace{\varphi\Bigl(\frac{1}{n}x\Bigr)+\ldots+\varphi\Bigl(\frac{1}{n}x\Bigr)}_{m\ \text{times}} = m\varphi\Bigl(\frac{1}{n}x\Bigr).
\end{equation}
Then, second,
\begin{equation}\nonumber
  \varphi(x) = \varphi\Bigl(\frac{n}{n} x\Bigr) = n\varphi\Bigl(\frac{1}{n}x\Bigr),
\end{equation}
whence
\begin{equation}\nonumber
  \varphi\Bigl(\frac{1}{n}x\Bigr) = \frac{1}{n}\varphi(x).
\end{equation}
And lastly, in combination with the first, this yields that
\begin{equation}\nonumber
  \varphi\Bigl(\frac{m}{n} x\Bigr) = \frac{m}{n} \varphi(x),
\end{equation}
that is,
\begin{equation}\nonumber
  \varphi(\lambda x) = \lambda\varphi(x).]
\end{equation}

Now, consider the case of non-prime field: let $F\ne k$. Then there exists an $f\in F$ such that $f\notin k$. Clearly, \mbox{$f\ne 0,\ 1$}, and the set $\{1,\ f\}$ is $k$-independent. Extend this set to a basis $B$ of the \mbox{$k$-space}~$F$ (assuming Axiom of Choice, Zorn's lemma or something like). Define $\varphi\colon B\to F$ by
\begin{equation}\nonumber
  \varphi(b)= f\quad \text{for all } b\in B.
\end{equation}
The mapping $\varphi$ has a $k$-linear extension, say $\tilde{\varphi}$, to the $k$-space $F$. This mapping,
$\tilde{\varphi}$, is an additive mapping of $F$-space $F$ to itself which is not $F$-homogeneous
(and therefore not $F$-linear).

In fact, on the one hand, $\tilde\varphi(1) = \tilde{\varphi}(f) = f$ (because $\tilde{\varphi}$ is
an extension of $\varphi$, and as such it coincides with $\varphi$ on $B$). But, on the other hand,
if $\tilde{\varphi}$ is $F$-homogeneous, then there must be $\tilde{\varphi}(f) = \tilde{\varphi}(f\cdot 1) = f\cdot
\tilde{\varphi}(1) = f\cdot f = f^2$.

However, $f^2\ne f$, because
\begin{equation}\nonumber
  f^2=f\leftrightarrow f^2-f=0 \leftrightarrow f\cdot(f-1)=0 \leftrightarrow f=0 \text{ or } f=1,
\end{equation}
what is not the case here.
\end{proof}

\subsection{From (H) to (A)}
Now we show that homogeneity does not imply additivity. The question,
\begin{quote}
  Does homogeneity imply additivity?
\end{quote}
was placed by the author in ResearchGate. For two years it reached 877 reads, 7 followers and 4 recomendations, and remained unanswered. Here we give full answer to the question.
\begin{theorem}
For every field $F$ there exist vector $F$-spaces $U$ and $V$ and a mapping $\varphi\colon U\to V$ that is $F$-homogeneous but not additive.
\end{theorem}

\begin{proof}
Let $F$ be arbitrary field. Take $U=V=F^2$ (two-dimensional arithmetic space over $F$). Let $\{e_1,\ e_2\}$ be standard basis of $U$: $e_1=(1,\ 0)$, $e_2=(0,\ 1)$. Define $\varphi\colon U\to V$ by the rule
\begin{equation}\nonumber
  \varphi(u)=\begin{cases}
   u, &\text{if } u\in F\cdot\{e_1,\ e_2\};\\
   (0,\ 0), &\text{otherwise}.
  \end{cases}
\end{equation}
\textbullet\ \ So defined $\varphi$ is $F$-homogeneous. In fact, let $u\in U$ and $\lambda\in F$. If $u\in F\cdot\{e_1,\ e_2\}$, then also $\lambda u\in F\cdot\{e_1,\ e_2\}$, and we have
\[
  \varphi(\lambda u) = \lambda u = \lambda\varphi(u).
\]
Now, let $u\notin F\cdot\{e_1,\ e_2\}$. Then $u=(\alpha_1,\ \alpha_2)$ with $\alpha_1\alpha_2\ne0$ and $\varphi(u)=(0,\ 0)$. At the same time, either $\lambda=0$, or $\lambda\ne 0$. If $\lambda=0$, then
\[
  \lambda u = (0,\ 0) = 0\cdot e_1\in F\cdot\{e_1,\ e_2\},
\]
and
\[
  \varphi(\lambda u) = \lambda u = (0,\ 0)=0\cdot(0,\ 0) = \lambda\varphi(u).
\]
At last, if $\lambda\ne 0$, then also $\lambda^2\ne 0$, and
\[
  0\ne\lambda^2\cdot\alpha_1\alpha_2 = (\lambda\alpha_1)(\lambda\alpha_2).
\]
This means that $\lambda u = (\lambda\alpha_1, \lambda\alpha_2)\notin F\cdot\{e_1,\ e_2\}$, so
\[
  \varphi(\lambda u) = (0,\ 0) = \lambda\cdot(0,\ 0)=\lambda\varphi(u).
\]
The $F$-homogeneity of $\varphi$ is proven.

\textbullet\ However, this $\varphi$ is not additive:
\begin{equation}\nonumber
  \varphi(e_1+e_2)=\varphi((1,\ 1))=(0,\ 0) \ne (1,\ 1) = (1,\ 0)+(0,\ 1)=e_1+e_2=\varphi(e_1)+\varphi(e_2).\qedhere
\end{equation}
\end{proof}

{\sc Gratitude.} The author is very grateful to Greg Oman of the U Colorado for comments and suggestions that have contributed to the improvement of the article.

\bibliographystyle{amsplain}

\end{document}